\documentclass{amsart}
\usepackage[margin=1.5in]{geometry}
\usepackage[normalem]{ulem}

\usepackage{tikz}
\usepackage[colorlinks=true,urlcolor=blue]{hyperref} 
\usepackage[numbers, sort&compress]{natbib}

\newtheorem{thm}{Theorem}[section]
\newtheorem{lem}[thm]{Lemma}
\newtheorem{prop}[thm]{Proposition}
\newtheorem{cor}[thm]{Corollary}
\newtheorem{dfn}[thm]{Definition}
\newtheorem{fact}[thm]{Fact}

\newtheorem{rmk}[thm]{Remark}

\newcommand{\refT}[1]{Theorem~\ref{#1}}
\newcommand{\refC}[1]{Corollary~\ref{#1}}
\newcommand{\refL}[1]{Lemma~\ref{#1}}

\newcommand{\refS}[1]{Section~\ref{#1}}
\newcommand{\refP}[1]{Proposition~\ref{#1}}

\newcommand{\refF}[1]{Figure~\ref{#1}}

\newcommand{\refand}[2]{\ref{#1} and~\ref{#2}}

\newcommand\cA{\mathcal A}

\newcommand\cF{\mathcal F}

\newcommand\cI{\mathcal I}

\newcommand\cP{\mathcal P}

\newcommand\cS{{\mathcal S}}

\newcommand{\E}[1]{{\mathbf E}\left[#1\right]}

\newcommand{\p}[1]{{\mathbf P}\left(#1\right)}
\newcommand{\I}[1]{{\mathbf 1}_{[#1]}}

\newcommand{\N}{\mathbb{N}}
\newcommand{\R}{\mathbb{R}}
\newcommand{\Z}{\mathbb{Z}}
\newcommand{\eps}{\varepsilon}

\title{High degrees in random recursive trees}
\author{Louigi Addario-Berry}
\email{louigi@problab.ca}
\author{Laura Eslava}
\email{laura.eslavafernandez@mail.mcgill.ca}
\address{Department of Mathematics and Statistics, McGill University, Montreal, Canada}
\date{July 4, 2016}

\begin{document}

\begin{abstract}
For $n\ge 1$, let $T_n$ be a random recursive tree (RRT) on the vertex set $[n]=\{1,\ldots,n\}$. 
Let $\deg_{T_n}(v)$ be the degree of vertex $v$ in $T_n$, that is, the number of children of $v$ in $T_n$. Devroye and Lu \cite{DevroyeLu95} showed that the maximum degree $\Delta_n$ of $T_n$ satisfies $\Delta_n/\lfloor \log_2 n\rfloor \to 1$ almost surely; Goh and Schmutz \cite{GohSchmutz02} showed distributional convergence of $\Delta_n - \lfloor \log_2 n \rfloor$ along suitable subsequences. In this work we show how a version of Kingman's coalescent can be used to access much finer properties of the degree distribution in $T_n$. 

For any $i\in \mathbb{Z}$, let $X_i^{(n)}=|\{v\in [n]: \deg_{T_n}(v)= \lfloor \log n\rfloor +i\}|$. Also, let $\mathcal{P}$ be a Poisson point process on $\mathbb{R}$ with rate function $\lambda(x)=2^{-x}\cdot \ln 2$. We show that, up to lattice effects, the vectors $(X_i^{(n)},\, i\in \mathbb{Z})$ converge weakly in distribution to $(\mathcal{P}[i,i+1),\, i\in \mathbb{Z})$. We also prove asymptotic normality of $X_i^{(n)}$ when $i=i(n) \to -\infty$ slowly, and obtain precise asymptotics for $\p{\Delta_n - \log_2 n > i}$ when $ i(n) \to \infty$ and $i(n)/\log n$ is not too large. Our results recover and extends the previous distributional convergence results on maximal and near-maximal degrees in random recursive trees. 
\end{abstract}

\maketitle

\section{Statement of results}
The process of random recursive trees $(T_n,\, n\ge 1)$ is defined as follows. $T_1$ has a single node with label 1, which its root. The tree $T_{n+1}$ is obtained from $T_n$ by directing an edge from a new vertex $n+1$ to $v\in [n]$; the choice of $v$ is uniformly random and independent for each $n\in \mathbb{N}$. We call $T_n$ a random recursive tree (RRT) of size $n$. 

As a consequence of the construction, vertex-labels in $T_n$ increase along root-to-leaf paths. Rooted labelled trees with such property are called \emph{increasing trees}. It is not difficult to see that, in fact, $T_n$ is uniformly chosen among the set $\cI_n$ of increasing trees with vertex set $[n]$.

We write $\deg_{T_n}(v)$ to denote the number of children of $v$ in $T_n$. 
The degree distribution of $T_n$ is encoded by the variables $Z_i^ {(n)}=|\{v\in [n]: \deg_{T_n}(v)=i\}|$, for $i \ge 0$. In fact, the study of RRT's started with a paper by Na and Rapoport \cite{NaRapoport70} in which they obtained, for any {\em fixed} $i\ge 0$, the convergence $\mathbb{E}(Z_i^ {(n)})/n\to 2^{-i-1}$ as $n\to \infty$; this result was extended to convergence in probability by Meir and Moon in \cite{MeirMoon88}. Mahmoud and Smythe \cite{MahmoudSmythe92} derived the asymptotic joint normality of $Z_i^ {(n)}$ for $i\in \{0,1,2\}$; and finally, Janson \cite{Janson05} extended the joint normality to $Z_i^ {(n)}$ for $i \ge 0$ and gave explicit formulae for the covariance matrix. 

The above results concern typical degrees; the focus in this work is large degree vertices, and in particular the maximum degree in $T_n$, which we denote $\Delta_n=\max_{v\in [n]} \deg_{T_n}(v)$. For the rest of the paper we write $\log$ to denote logarithms with base 2, and $\ln$ to denote natural logarithms. For $n\in \N$ let  $\eps_n=\log n-\lfloor \log n\rfloor$.

A heuristic to find the order of $\Delta_n$ is that, if $\mathbb{E}(Z_i^ {(n)})\approx n2^{-i-1}$ were to hold for all $i$, as it does when $i$ is fixed, then we would have $\mathbb{E}(Z_{\lfloor \log n \rfloor}^ {(n)})\approx n2^{-\lfloor \log n \rfloor-1}=2^{-1+\eps_n}$. This heuristic suggests that $\Delta_n$ is of order $\log n$. This is indeed the case: Szymanski \cite{Szymanski90} proved that $\E{\Delta_n}/\log n\to 1$ as $n\to \infty$, and Devroye and Lu \cite{DevroyeLu95} later established that $\Delta_n/\log n\to 1$ a.s.. Finally, Goh and Schmutz \cite{GohSchmutz02} showed that $\Delta_n-\lfloor \log n\rfloor$ converges in distribution along suitable subsequences, and identified the possible limiting laws. 

Since we focus on maximal degrees, it is useful to let 
\[X_i^{(n)}=Z_{i+\lfloor \log n \rfloor}^{(n)} = |\{v\in [n]: \deg_{T_n}(v)= \lfloor \log n\rfloor +i\}|,\]
for $n \in \N$ and $i \ge -\lfloor \log n \rfloor$. The following is a simplified version of one of our main results. 

\begin{thm}\label{mainv0}
Fix $\varepsilon\in [0,1]$. Let $(n_l)_{l\ge 1}$ be an increasing sequence of integers satisfying $\varepsilon_{n_l} \to \varepsilon$ as $l\to \infty$. Then, as $l\to \infty$
\[(X_i^{(n_l)},\, i\in \Z)\stackrel{\mathrm{d}}{\longrightarrow} (P^{\eps}_i,\, i\in \Z) \]
jointly for all $i\in \Z$ where the $P^\eps_i$ are independent Poisson r.v.'s with mean $2^{-i-1+\eps}$.
\end{thm}

The random variables $X_i^{(n)}$ do not converge in distribution as $n\to \infty$ without taking subsequences; this is essentially a lattice effect caused by the floor $\lfloor \log n \rfloor$ in the definition of $X^{(n)}_i$.

\refT{mainv0} can be stated in terms of weak convergence of point processes (which is equivalent to convergence of finite dimensional distributions (FDD's); see Theorem 11.1.VII in \cite{DaleyVere-JonesII}). In fact, we will also prove convergence (along subsequences) of 
\[X_{\ge i}^{(n)}=\sum_{k\ge i} X_k^{(n)}=|\{v\in [n]: \deg_{T_n}(v)\ge \lfloor \log n \rfloor +i\}|.\]

This is useful as it yields information about $\Delta_n$ which cannot be derived from \refT{mainv0}. 
We formulate this result as a statement about convergence of point processes, and now provide the relevant definitions. 
Let $\Z^*=\Z\cup \{\infty\}$. Endow $\Z^*$ with the metric defined by $\mathrm{d}(i,j)=|2^{-j}-2^{-i}|$ and $\mathrm{d}(i,\infty)=2^{-i}$ for $i,j\in \Z$. Let $\mathcal{M}_{\Z^*}^{\#}$ be the space of boundedly finite measures of $\Z^*$.

Let $\mathcal{P}$ be a Poisson point process on $\mathbb{R}$ with rate function $\lambda(x)=2^{-x}\cdot \ln 2$. For each $\eps \in [0,1]$ let $\cP^\eps$ be the point process on $\Z^*$ given by
\[
\cP^{\eps}=\sum_{x\in \cP} \delta_{\lfloor x+\eps \rfloor}.
\]
Similarly, for all $n\in \N$ let 
\[
\cP^{(n)} =\sum_{v\in [n]} \delta_{\deg_{T_n}(v)-\lfloor \log n \rfloor}.
\]
Then, for each $i\in \Z$ we have that 
\[\cP^{\eps}(\{i\}):=|\{x\in \cP: \lfloor x+\eps\rfloor=i\}|=|\{x\in \cP: x\in [i-\eps,i+1-\eps)\}|\]
 has distribution $\mathrm{Poi}(2^{-i-1+\eps})$; also $\cP^{(n)}(\{i\})=X_i^{(n)}$. We abuse notation by writing, e.g., $\cP^{(n)}(i)=\cP^{(n)}(\{i\})$. 

It is clear that $\cP^{(n)}$ and $\cP^\eps$ are elements of $\mathcal{M}_{\Z^*}^{\#}$. 
The advantage of working on the state space to $\Z^*$ is that intervals $[k,\infty]$ are compact. In particular, the convergence of FDD's of $\cP^{(n_l)}$ implies the convergence in distribution of $X_{\ge i}^{(n_l)}=\cP^{(n_l)}[i,\infty)$. 

\begin{thm}\label{main}
Fix $\varepsilon\in [0,1]$. Let $(n_l)_{l\ge 1}$ be an increasing sequence of integers satisfying $\varepsilon_{n_l} \to \varepsilon$ as $l\to \infty$. Then in $\mathcal{M}_{\Z^*}^{\#}$, $\cP^{(n_l)}$ converges weakly to $\cP^\varepsilon$ as $l\to \infty$. Equivalently, for any $i<i'\in \Z$, jointly as $l\to \infty$ 
\[(X_i^{(n_l)},\ldots, X_{i'-1}^{(n_l)}, X_{\ge i}^{(n_l)}) \stackrel{\mathrm{d}}{\longrightarrow} (\cP^\eps(i),\ldots, \cP^\eps(i'-1), \cP^\eps[i',\infty)).\]
\end{thm}

Note that \refT{mainv0} follows from \refT{main}. We finish this section stating two additional results. The first is an extension of the main theorem from \cite{GohSchmutz02}, that result being essentially the case $i=O(1)$.

\begin{thm}\label{max}
For any $i=i(n)$ with $i+\log n< 2 \ln n$ and $\liminf_{n\to \infty} i(n)>-\infty$, 
\[
\mathbf{P}(\Delta_n \ge \lfloor \log n\rfloor +i)=(1-\exp\{-2^{-i+\varepsilon_n}\})(1+o(1)).
\]
\end{thm}

When $i=O(1)$, the assertion of \refT{max} is a straight-forward consequence of \refT{main}. For the case that $i(n)\to \infty$ we use estimates for the first and second moments of $X_{\ge i}^{(n)}$; note that $\{\Delta_n< \lfloor \log n\rfloor +i\}=\{X_{\ge i}^{(n)}=0\}$.
 
Finally, we also obtain the asymptotic normality for $X_i^{(n)}$ when $i$ tends to $-\infty$ slowly enough.  

\begin{thm}\label{normal}
If $i=i(n)\to -\infty$ and $i=o(\ln n)$, then as $n\to \infty$
\[\frac{X_i^{(n)}-2^{-i-1+\eps_n}}{\sqrt{2^{-i-1+\eps_n}}}\stackrel{\mathrm{d}}{\to} N(0,1).\]
\end{thm}

\begin{rmk}
Up to lattice effects, Theorems \refand{main}{normal} extend the range of $i=i(n)$ for which the heuristic that $Z_i^{(n)} \approx n2^{-i-1}$ holds.
\end{rmk}

A key novelty of our approach is that for each $n$ we use \emph{Kingman's coalescent} to generate a tree $T^{(n)}$ whose vertex degrees $\{\deg_{T^{(n)}}(v)\}_{v\in[n]}$ are exchangeable but otherwise have the same law as degrees in $T_n$. (See \cite{Berestycki09}, Chapter 2 for a description of Kingman's coalescent, and \cite{ab14}, Section 2.2 for a description of the connection with random recursive trees which we exploit in this paper.)
By this we mean that if $\sigma:[n]\to[n]$ is a uniformly random permutation then the following distributional identiy holds: 
\begin{equation}\label{relabel}
(\deg_{T^{(n)}}(v),\, v\in [n])\stackrel{\mathrm{d}}{=}(\deg_{T_n}(\sigma (v)),\, v\in [n]). 
\end{equation}
We describe the trees $T^{(n)}$, $n \in \N$ in Section~\ref{model}. 

An essentially equivalent construction was used by Devroye \cite{Devroye87} to study union-find trees. In \cite{Pittel94}, Pittel related the results of \cite{Devroye87} on union-find trees to the height of RRT's. It is worth mentioning that both Kingman's coalescent and the union-find trees can be equivalently represented as binary trees or, as we will see in \refS{model}, as RRT's.
Aside from the works \cite{Devroye87} and \cite{Pittel94}, it seems that the use of Kingman's coalescent or of union-find trees to study RRT's is rare. However, it turns out to provide just the right perspective for studying high degree vertices. 

\section{Outline}\label{sketch}

In this section we sketch the approach used in the paper. The proofs of the theorems relay on the computation of the moments of the FDD's of $\cP^{(n)}$; these estimates are given in \refP{moments}. In particular, the proofs of Theorems \refand{main}{normal} use the method of moments (e.g., see \cite{JLR} Section 6.1, and \cite{Bollobas} Section 1.5). 

Any FDD of $\cP^{(n)}$ can be recovered from suitable marginals of the joint distribution of $(X_i^{(n_l)},\ldots, X_{i'-1}^{(n_l)}, X_{\ge i'}^{(n_l)})$ for some $i<i'\in \Z$. For simplicity, we focus for the moment on collections of variables $X_{i}^{(n)},\ldots,X_{i'}^{(n)}$ for $i\le i'$. For $r\in \R$ and $a\in \N$ write $(r)_a=r(r-1)\cdots (r-a+1)$, also let $(r)_0=1$. 
We will prove that for any non-negative integers $a_i,\ldots,a_{i'}$, as $n\to \infty$, we have 

\begin{equation}\label{moments limit}
\E{\prod_{i\le k\le i'} (X_{k}^{(n)})_{a_k}}- \prod_{i\le k\le i'} \left( 2^{-(k+1)+\eps_n} \right)^{a_k} \to 0.
\end{equation}
This immediately yields \refT{mainv0}. 

By the linearity of expectation, proving \eqref{moments limit} reduces to understanding the probabilities
\begin{equation}\label{non-exch}
\p{\deg_{T_n}(v_k)=\lfloor \log n \rfloor +i_k,\, k\in [K]}
\end{equation}
for all $i_1,\ldots i_K\in \N$ and $v_1, \ldots v_K\in [n]$, $K\in \N$; see \refS{proofmoments} for more details.

In the standard model for RRT's described at the beginning, 
$\deg_{T_n}(v)$ is a sum of Bernoulli variables: 
\[\deg_{T_n}(v)=\sum_{v< u \le n} \mathbf{1}_{\{u\to v\}}.\] 
The lack of symmetry of the degrees $\{\deg_{T_n}(v)\}_{v\in [n]}$ complicates the analysis of \eqref{non-exch}. 
In proving that $\Delta_n/\log n \stackrel{\mathrm{a.s.}}{\rightarrow} 1$, Devroye and Lu \cite{DevroyeLu95} used that $\{\deg_{T_n}(v)\}_{v\in [n]}$ are negatively orthant dependent (see \cite{Joag-DevProschan83} for a definition), which in particular means that for all $S\subset [n]$ and $m_1,\ldots, m_n\in \N$
\begin{equation}\label{orthant}
\p{\deg_{T_n}(v)\ge m_v,\, v\in S}\le \prod_{v\in S} \p{\deg_{T_n}(v)\ge m_v}
\end{equation}
and then obtained upper bounds for $\p{deg_{T_n}(v)\ge c\ln n}$ for each $v\in [n]$. 

One approach to studying high degrees in $T_n$ would be to obtain matching lower bounds for $\p{\deg_{T_n}(v)\ge m_v,\, v\in S}$, with uniform error terms even when $m_v$ is large. Instead, we study trees $T^{(n)}$, mentioned in \eqref{relabel}, above, for which we can obtain precise asymptotics for the analogous probabilities
\begin{equation}\label{mainprob}
\p{\deg_{T^{(n)}}(v)\ge m_v,\, v\in [K]}.
\end{equation}

The core of the paper lies in \refP{core prop}, which gives precise estimates of \eqref{mainprob} for $m_1, \ldots,m_K$ in a suitable range. Broadly speaking, $\deg_{T^{(n)}}(v)$ depends on a set of random \emph{selection times} $\cS_v$ and the first streak of heads in a sequence of $|\cS_v|$ fair coin flips. As mentioned in the previous section, the degrees of $T^{(n)}$ have the same distribution as the degrees in $T_n$. Consequently, our estimation of \eqref{mainprob} allows us to obtain the following moments estimate.

\begin{prop}\label{moments}
For all $c\in (0,2)$ and $K\in \N$ there is $\alpha=\alpha(c,K)>0$ such that the following holds. Fix any integers $i,i'$ with $0<i+\log_n<i'+\log_n<c\ln n$. Then for any non-negative integers $a_i,\ldots, a_{i'}$ with $a_i+\ldots+a_{i'}=K$, we have
\begin{align*}
\E{(X_{\ge i'}^{(n)})_{a_{i'}} \prod_{i\le k<i'} (X_{k}^{(n)})_{a_k}}
&=\left(2^{-i'+\eps_n} \right)^{a_{i'}} \prod_{i\le k<i'} \left( 2^{-(k+1)+\eps_n} \right)^{a_k}(1+o(n^{-\alpha})) .
\end{align*}
\end{prop} 
Equipped with \refP{moments}, the proofs of the theorems are straightforward. 
The rest of the paper is organized as follows. In \refS{model}, we explain how to define the trees $T^{(n)}$ using Kingman's coalescent and establish the distributional relation between $T^{(n)}$ and the RRT; see \refC{relabel cor}.
In \refS{core}, we define the random sets $(\cS_v, v\in T^{(n)})$ and explain their relation with degrees in $T^{(n)}$. The proof of \refP{core prop}, which is our estimate of \eqref{mainprob}, is  then presented using a decoupling of the events in \eqref{mainprob} and the  concentration of the random variables $|\cS_v|$. Finally, the proof of \refP{moments} is given in \refS{proofmoments} and the proof of Theorems \ref{main}-\ref{normal} are in \refS{proofs}.

\section{Random Recursive Trees and Kingman's coalescent}\label{model}

In this section we give a representation of Kingman's coalescent in terms of labelled forests, and relate it to RRT's. All trees in the remainder of the paper are rooted, and we write $r(t)$ for the root of tree $t$. By convention, edges of a tree are directed towards the root of the tree and we write $uv$ to denote an edge directed from $u$ to $v$. A forest $f$ is a set of trees whose vertex sets are pairwise disjoint. The vertex set of a forest, denoted $V(f)$, is the union of the vertex sets of its trees. Similarly, $E(f)$ denotes the set of edges in the trees of $f$. For $n\ge 1$, let 
\[\cF_n=\{f:V(f)=[n]\}\] be the set of forests with vertex set $[n]$. 

A sequence $C=(f_1,\ldots,f_n)$ of elements of $\cF_n$ is an $n$-chain if $f_1$ is the forest in $\cF_n$ with $n$ one-vertex trees and, for $1\le i<n$, $f_{i+1}$ is obtained from $f_i$ by adding a directed edge between the roots of some pair of trees in $f_i$. 
If $(f_1, \ldots, f_n)$ is an $n$-chain then for $1\le i\le n$, the forest $f_i$ consists of $n+1-i$ trees, and in this case we list its elements in increasing order of their smallest-labelled vertex as $t_1^{(i)}, \ldots, t_{n+1-i}^{(i)}$.

\begin{dfn} 
Kingman's $n$-coalescent is the random $n$-chain $\mathbf{C}=(F_1,\ldots, F_n)$ built as follows. 
Independently for each $1\le i\le n-1$ let $\{a_i,b_i\}$ be a random pair uniformly chosen from $\{\{a,b\}: 1\le a<b\le n+1-i\}$ and let $\xi_i$ be independent with $\mathrm{Bernoulli}(1/2)$ distribution. 

For $1\le i<n$, construct $F_{i+1}$ from $F_i$ as follows. If $\xi_i=1$ then add an edge from $r(T_{b_i}^{(i)})$ to $r(T_{a_i}^{(i)})$ and if $\xi_i=0$ then add an edge from $r(T_{a_i}^{(i)})$ to $r(T_{b_i}^{(i)})$. The forest $F_{i+1}$ consists of the new tree and the remaining $n-1-i$ unaltered trees from $F_i$.
\end{dfn}

For an example of the process see \refF{chain}.

\begin{figure} \begin{center}
\includegraphics[scale=.8]{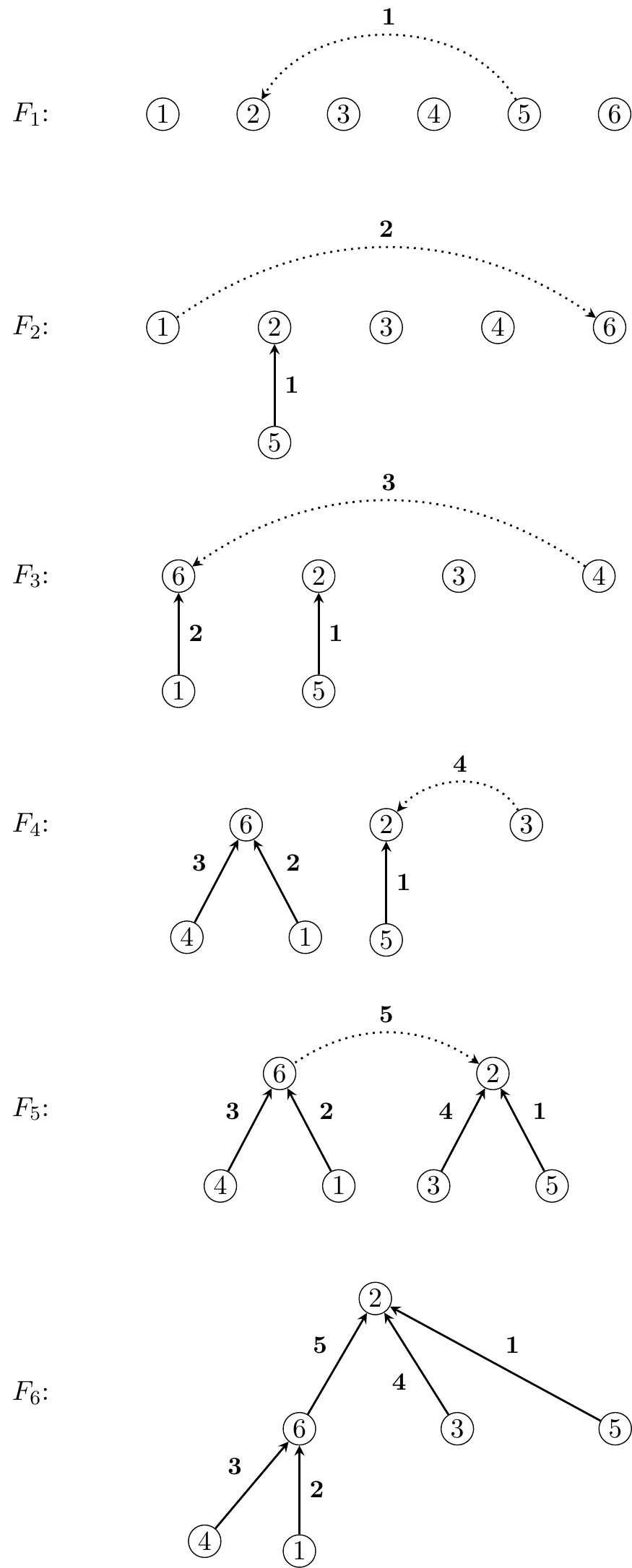}
\end{center}
\caption{An example of Kingman's $n$-coalescent $\mathbf{C}=(F_1,\ldots,F_n)$ for $n=6$. For $1\le i<n$, $F_i$ has, in dotted line, the edge in $E(F_{i+1})\setminus E(F_i)$. Edges are marked with their time of addition; this is the function $L^-_{\mathbf{C}}$ defined after \refL{uniform chain}.
In this instance, $\xi_1=\xi_3=\xi_4=1$, $\xi_2=\xi_5=0$ and $\{a_1,b_1\}=\{2,5\},$ $\;\{a_2,b_2\}=\{1,5\},\;\{a_3,b_3\}=\{1,4\},\; \{a_4,b_4\}=\{2,3\},\;\{a_5,b_5\}=\{1,2\}$.}
\label{chain}
\end{figure}

\begin{lem}\label{uniform chain}
Let $\mathcal{CF}_n$ be the set of $n$-chains of elements in $\cF_n$. Then $|\mathcal{CF}_n|=n!(n-1)!$ and Kingman's $n$-coalescent  is a uniformly random element of $\mathcal{CF}_n$.
\end{lem}
\begin{proof}
Fix an $n$-chain $(f_1,\ldots, f_n)\in \mathcal{CF}_n$. Then 
\[\p{(F_1,\ldots,F_n)=(f_1,\ldots,f_n)}=\prod_{k=1}^{n-1} \p{F_{k+1}=f_{k+1}|F_{j}=f_j,\, 1\le j\le k}.\]
Among the $(n+1-k)(n-k)$ possible oriented edges between roots of $f_k$, there is exactly one whose addition yields $f_{k+1}$. It follows that the $k$-th term in the above product is $\left((n+1-k)(n-k)\right)^{-1}$, so 
$\p{(F_1,\ldots,F_n)=(f_1,\ldots,f_n)}=\left[n!(n-1)!\right]^{-1}$. 
The result follows since this expression does not depend on $(f_1,\ldots, f_n)\in \mathcal{CF}_n$. 
\end{proof}

Recall that $\cI_n$ is the set of increasing trees with vertex set $[n]$. It is not difficult to see that $|\cI_n|=(n-1)!$ and that a RRT is a uniformly random element of $\cI_n$. 

There is a natural mapping $\phi$ between $n$-chains and increasing trees. Given an $n$-chain $C=(f_1,\ldots, f_n)$, write $t^{(n)}:=t_1^{(n)}$ for the unique tree in $f_n$. 
Let $L^-_C:E(t^{(n)})\to [n-1]$ be defined as follows. For each $e\in E(t^{(n)})$, let 
\[L^-_C(e)=\max\{i\in [n-1]:\, e\notin E(t^{(i)})\}.\]
We think of $L^-_C$ as a function that keeps track of the \emph{time of addition} of the edges along the $n$-chain $C$.
Now, we define a vertex labelling $L_C:V(t^{(n)})\to [n]$ as follows. Let $L_C(r(t^{(n)}))=1$ and for each $uv\in E(t^{(n)})$, let 
\[L_C(u)=n+1-L^-_C(uv);\] 
then $L_C(u)$ is the number of trees in the forest just before $uv$ is added.

Note that for each $i\in [n-1]$, the new edge in $f_{i+1}$ joins the roots of two trees in $f_i$ and is directed towards the root of the resulting tree.
Thus, the labels $\{L^-_C(e),\, e\in E(t^{(n)})\}$ increase along all paths in $t^{(n)}$ towards the root $r(t^{(n)})$ and consequently, the labels $\{L_C(v),\, v\in V(t^{(n)})\}$ increase along root-to-leaf paths in $t^{(n)}$. 
This shows that relabelling the vertices of $t^{(n)}$ with $L_C$ yields an increasing tree (specifically, an element of $\cI_n$). 
See Figure \ref{relabel fig} for an example. 

\begin{figure} \begin{center}
\includegraphics[scale=.9]{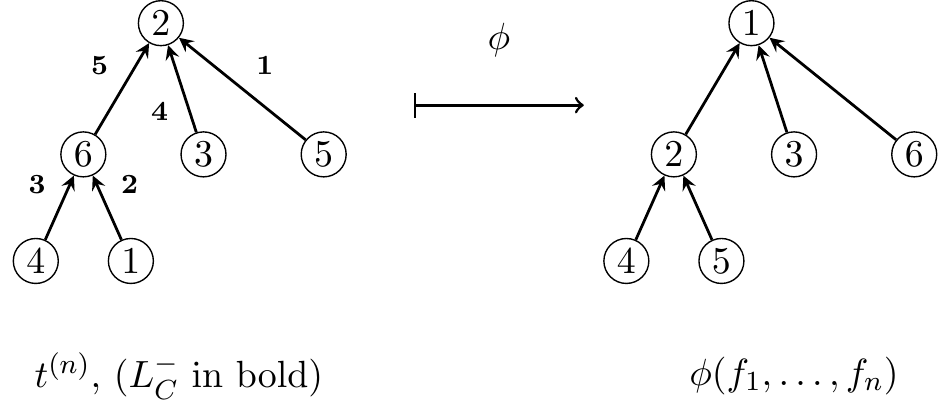}
\end{center}
\caption{On the left a tree $t^{(n)}$; edges are marked with $L^-_C$, from which the $n$-chain $C=(f_1,\ldots, f_n)$ can be recovered. On the right, the increasing tree $\phi(f_1,\ldots, f_n)$; it has the shape of $t^{(n)}$ and the vertex labels $\{L_C(v),\, v\in V(t^{(n)})\}$.}
\label{relabel fig}
\end{figure}

\begin{prop}\label{law}
Let $\phi:\mathcal{CF}\to \cI_n$ be defined as follows. For an $n$-chain $C=(f_1,\ldots,f_n)$ let $\phi(C)$ be the tree obtained from $t^{(n)}$ by relabelling its vertices with $L_C$. Then $\phi(\mathbf{C})$, the push-forward of Kingman's $n$-coalescent by $\phi$, has the law of a RRT of size $n$. 
\end{prop}
\begin{proof}
First, we prove that $\phi$ is onto. Fix an increasing tree $t\in \cI_n$. For each $j\in V(t)\setminus \{1\}$, let $v_j\in V(t)$ be such that $jv_j\in E(t)$, recall that edges are directed toward the root of $t$, thus $v_j$ is uniquely defined. For each $1<j\le n$, let $e_{n-j+1}=ju_j$. 

Now construct an $n$-chain $C$ as follows. Let $f_1$ be the forest with $n$ one-vertex trees. For each $1<i\le n$ construct $f_{i}$ from $f_{i-1}$ by adding the edge $e_{i-1}$. In other words, for each $1\le i<n$, $L^-_C(e_i)=i$ and so $L_C(n+1-i)=n+1-L^-_C(e_i)=n+1-i$; also since $r(t)=1$, we have $L_C(1)=1$. Consequently, $\phi(C)=t$. 

We claim that $|\phi^{-1}(t)|\ge n!$ for any $t\in \cI_n$. To see this, consider an $n$-chain $C$ and a permutation $\sigma:[n]\to [n]$. Let $C_\sigma$ be the $n$-chain obtained from $C$ by permuting the vertices in each forest of $C$ by $\sigma$. 
Since $L_C(v)$ depends only on the time of addition of its outgoing edge (if any), it follows that $\phi(C)=\phi(C_\sigma)$ for all permutations $\sigma$. 
By \refL{uniform chain}, this shows that $\phi$ is $n!$-to-1 and that $\phi(\mathbf{C})$ is a uniform element in $\cI_n$. 
\end{proof}

Since $\phi(\mathbf{C})$ preserves the shape of $T^{(n)}$ and only relabels its vertices, the degrees in $T^{(n)}$ and $\phi(\mathbf{C})$ are equal as multisets:
$\{deg_{T^{(n)}}(v)\}_{v\in [n]}=\{deg_{\phi(\mathbf{C})}(v)\}_{v\in [n]}$. 
This immediately gives the following key corollary of Proposition~\ref{law}, on which the rest of the paper relies. 
\begin{cor}\label{relabel cor}
For all $n\in \N$, we have the following equality in distribution holds jointly for all $i\in \Z$, 
\begin{equation*}
X_i^{(n)}\stackrel{\mathrm{d}}{=}|\{v\in [n]: \deg_{T^{(n)}}(v)= \lfloor \log n\rfloor +i\}|.
\end{equation*}
\end{cor} 

We now proceed to the study of the joint distribution of the vertex degrees in $T^{(n)}$. 

\section{Degree distribution: Selection sets and coin flips}\label{core}

By construction, the vertex degrees $\{\deg_{T^{(n)}}(v)\}_{v\in [n]}$ are exchangeable. Our next goal is to explain how to approximate \eqref{mainprob}; that is, for any fixed $k\in \N$ and integers $m_1,\ldots, m_k< 2\ln n$, to obtain estimates for $\p{\deg_{T^{(n)}}(v)\ge m_v,\, v\in [k]}$.
 
The key to analyse the degrees in $T^{(n)}$ is to understand how the degrees of a vertex $v\in [n]$ change in Kingman's coalescent $\mathbf{C}=(F_1,\ldots, F_n)$. For any vertex $v$ and $1\le i\le n$, denote $\deg_{F_i}(v)$ the number of children of $v$ in $F_i$. Also, we will simply write $\deg(v)=\deg_{F_n}(v)=\deg_{T^{(n)}}(v)$. For each $1\le i<n$, if $\xi_i=1$ we say that $\xi_i$ \emph{favours} the vertices of $T_{a_i}^{(i)}$, and otherwise that it favours the vertices of $T_{b_i}^{(i)}$. For $v\in [n]$, let 
\[\cS_v=\{i\in [n-1]: v\in T_{a_i}^{(i)}\cup T_{b_i}^{(i)}\}.\]	

For any vertex $v$, and $1\le i<n$, $\deg_{F_{i+1}}(v)$ increases by one only if $v$ is a root in $F_i$, $i\in \cS_v$ and $\xi_i$ favours $v$; see Figure \ref{step t}. Conversely, let $p_v=\min\{i\in \cS_v,\, \xi_i \text{ does not favour } v\}$, then the first $F_{i+1}$ in which $v$ is not a root is exactly $i=p_v$. In this case, in $F_{p_v+1}$ there is an outgoing edge from $v$, and $v$ is not a root of any subsequent forests. As a consequence, $\deg_{F_{j}}(v)=\deg_{F_{p_v}}(v)$ for $p_v< j\le n$.

\begin{figure} \begin{center}
\includegraphics[scale=.9]{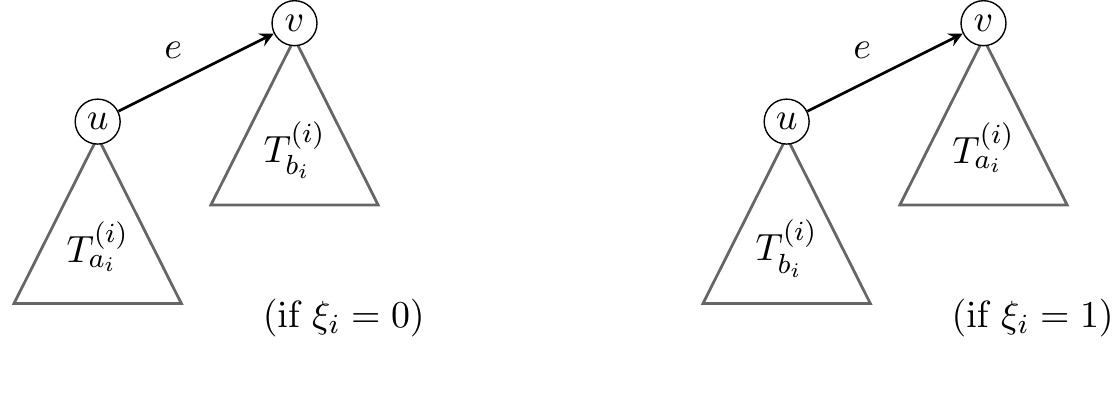}
\end{center}
\caption{If $v$ is a root in $T_{a_i}^{(i)}\cup T_{b_i}^{(i)}$ and $\xi_i$ favours $v$, then $v$ increases its degree and remains a root in $F_{i+1}$.}
\label{step t}
\end{figure}

\begin{fact}\label{streak}
For $v\in [n]$, $\deg(v)=\deg_{F_{p_v}}(v)=|\cS_v\cap [p_v-1]|$.
\end{fact}

In other words, $\deg(v)$ depends only on its first streak of favourable random variables $\xi_i$ with $i\in \cS_v$. More precisely, given $|\cS_v|$, the degree $\deg(v)$ is distributed as $\min\{|\cS_v|,G\}$, where $G$ is a $\mathrm{Geometric}(1/2)$ r.v. independent of $\cS_v$. 

Thus, it is relevant to observe that $|\cS_v|$ is distributed as an sum of independent (though not identically distributed) Bernoulli random variables and so it is concentrated around its mean $\E{|\cS_v|}=2\ln n+O(1)$; a more precise statement can be found in \refP{bound Ik} below. Since $|\cS_v|\to \infty$ in probability as $n\to \infty$, it follows easily that $\deg(v)$ is asymptotically geometric for any fixed node $v$.
 More strongly, the following proposition shows that for any fixed $k$, the random variables $\{\deg_{T^{(n)}}(v)\}_{v\in [k]}$ asymptotically behave like independent Geometric random variables, even if they are conditioned to be quite large.

\begin{prop}\label{core prop}
Fix $c\in (0,2)$ and $k\in \N$. There exists $\alpha=\alpha(c,k)>0$ such that uniformly over positive integers $m_1,\ldots, m_k< c\ln n$,
\[\p{\deg_{T^{(n)}}(v)\ge m_v,\, v\in [k]}= 2^{-\sum_{v} m_v} (1+o(n^{-\alpha})).\]
\end{prop}

We now explain how the events in the proposition above can be decoupled into a product of two probabilities, one of them corresponding to tail bounds for the random variables $|\cS_v|$. We start with an upper bound for \refP{core prop}. 
 
\begin{lem}\label{Upper bound}
For any $k \in \N$ and positive integers $m_1,\ldots, m_k<n$,
\[\p{\deg(v)\ge m_v,\, v\in [k]}\le 2^{-\sum_{v} m_v}\,\p{|\cS_v|\ge m_v,\, v\in [k]}.\]
Equality holds for $k=1$. 
\end{lem}

\begin{proof}
For each $v\in [k]$ list $\cS_v$ in increasing order as $(i_{v,j},\, 1\le j\le |\cS_v|)$. Let $\cA$ be the set of sequences $A=(A_1,\ldots, A_k)$ satisfying $A_v\subset [n-1]$  and $|A_v|=m_v$ for all $v\in [k]$. For every $A\in \cA$, let $D_A$ be the event that $|\cS_v|\ge m_v$ and $\{i_{v,1},\ldots, i_{v,m_v}\}=A_v$, for all $v\in [k]$. By Fact \ref{streak}, if 
$\deg(v)\ge m_v$ then necessarily $|\cS_v|\ge m_v$ so
\begin{align*}
\{\deg(v)\ge m_v,\, v\in [k]\} \cap D_A
=\{\xi_{i_{v,j}} \text{ favours $v$ for all } j\in [m_v],\, v\in [k]\} \cap D_A. 
\end{align*}
Now, $\xi_i$ are i.i.d Bernoulli(1/2) r.v.'s. Thus, if $D_A$ has positive probability then
\begin{equation*}
\p{\xi_{i_{v,j}} \text{ favours $v$ for all } j\in [m_v],\, v\in [k]|D_A}=
\begin{cases}
2^{-\sum_v m_v} & \text{ if $|A_u\cap A_v|=0,\, \forall u\neq v\in [k]$}\\
0  				& \text{ o.w.}
\end{cases}
\end{equation*}

The second case follows from the fact that if $i\in \cS_u\cap \cS_{v}$ for some $u\neq v$,  then $\xi_i$ cannot favour both $u$ and $v$. The events $(D_A,\, A\in \cA)$ are pairwise disjoint, and if $\deg(v)\ge m_v$ for all $v\in [k]$ then one of the events $D_A$ must occur. It follows that 
\begin{align*}
\p{\deg(v)\ge m_v,\, v\in [k]}
= & \sum_{A\in \cA}
\p{D_A ,\, \deg(v)\ge m_v,\, v\in [k]}\\
\le  & \sum_{A\in \cA}
 2^{-\sum_v m_v} \, \p{D_A}\\
= &  2^{-\sum_v m_v} \, \p{|S_v|\ge m_v,\, v\in [k]}.
\end{align*}
Finally, the second line holds with equality when $k=1$.
\end{proof}

For the lower bound we restrict to events $D_A$ where the sets $A_v$ are already disjoint. To do so, we consider instead the vertex degrees in $F_I$ for some $I<n$. For $k\ge 2$ let 
\[\tau_k=\min \{i\in [n-1]:\{a_i,b_i\}\subset [k]\}.\]
Since $F_i\subset F_{j}$ for all $i\le j\in [n]$ we have that for any $I<n$ 
\begin{align}\label{subset F}
\p{\deg(v)\ge m_v,\, v\in [k]}
&\ge \p{\deg_{F_{I+1}}(v)\ge m_v,\, v\in [k]} \nonumber\\
&\ge \p{I<\tau_k,\,\deg_{F_{I+1}}(v)\ge m_v,\, v\in [k]}.
\end{align}

Recall that trees in $F_i$ are listed in increasing order of their least elements; this implies that indices of the trees of vertices $1,\ldots, k$ do not change until two trees indexed by $a,b\le k$ are merged. Therefore, for all $v\in [k]$, $v\in T_v^ {(i)}$ for $i\le \tau_k$. This implies the sets $\{\cS_v\cap [\tau_k-1],\, v\in [k]\}$ are pairwise disjoint.
These observations allow us to obtain a lower bound analogous to \refL{Upper bound}. 

\begin{lem}\label{Lower bound}
For any positive integers $k\ge 2$ and $m_1,\ldots, m_k, I<n$, 
\[\p{\deg(v)\ge m_v,\, v\in [k]}\ge 2^{-\sum_v m_v}\p{I<\tau_k, |S_v\cap [I]|\ge m_v, v\in [k]}.\]
\end{lem}

\begin{proof} 
By \eqref{subset F}, it suffices to bound $\p{I<\tau_k,\, \deg_{F_{I+1}}(v)\ge m_v,\, v\in [k]}$. 

Let $\cA^*$ be the set of sequences $A=(A_1,\ldots, A_k)$ of pairwise disjoint subsets of $[I]$ satisfying $|A_v|=m_v$ for all $v\in [k]$. For each $A\in \cA^*$, let $D_A$ be the event that for all $v\in [k]$, $\{i_{v,j},\ldots, i_{v,m_v}\}=A_v$ (and so $|\cS_v\cap [I]|\ge m_v$). 

As in the proof of \refL{Upper bound}, we have that 
\begin{align*}
\{\deg_{F_{I+1}}(v)\ge m_v,\, v\in [k]\} \cap D_A
= \{\xi_{i_{v,j}} \text{ favours $v$ for all } j\in [m_v],\, v\in [k]\} \cap D_A.
\end{align*}
In this case, the sets $A_v$ are pairwise disjoint. If $\p{D_A}>0$ then
\begin{equation*}\label{favourT}
\p{\xi_{i_{v,j}} \text{ favours $v$ for all } j\in [m_v],\, v\in [k]|D_A}=2^{-\sum_v m_v}. 
\end{equation*}

Recall that $I<\tau_k$ if and only if the sets $\{S_v\cap [I],\, v\in [k]\}$ are pairwise disjoint; that is, if one of the events $D_A$ occur. We then have
\begin{align*}
\p{I<\tau_k,\, \deg_{F_{I+1}}(v)\ge m_v,\, v\in [k]}
= & \sum_{A\in \cA^*}
\p{D_A,\, \deg_{F_{I+1}}(v)\ge m_v,\, v\in [k]}\\
=  & \sum_{A\in \cA^*}
 2^{-\sum_v m_v} \, \p{D_A}\\
= & 2^{-\sum_v m_v} \, \p{I<\tau_k,\,|S_v\cap [I]|\ge m_v,\, v\in [k]}.
\end{align*}
\vspace{-1cm}

\end{proof}

To use \refL{Lower bound} we need tail bounds for $|\cS_v\cap [I]|$ for some suitable $I<n$; these are provided by the following proposition. 

\begin{prop}\label{bound Ik}
Fix $\eps\in (0,1)$ and $c \in (0,2(1-\eps))$. Then there exists $\beta=\beta(c,\eps)>0$ such that for any vertex $v$, 
\[\p{|\cS_v \cap[n-\lceil n^\eps\rceil]|<c\ln n}=o(n^{-\beta}). \]
\end{prop}

\begin{proof}
Fix $\eps \in (0,1)$ and $c \in (0,2(1-\eps))$.
Let $\{B_i,\, i\in \N\}$ be a collection of independent Bernoulli r.v.'s, with $\E{B_i}=\frac{2}{i}$. Recall the definition of $\cS_v$ at the beginning of the section. 

For any fixed vertex $v \in [n]$, and each $i\in [n-1]$, the probability of the event $\{v\in T_{a_i}^{(i)}\cup T_{b_i}^{(i)}\}$ is $2/(n-i+1)$; this is because, in the forest $F_i$, there are $n-i+1$ trees and the trees $T_{a_i}^{(i)},\, T_{b_i}^{(i)}$ are chosen uniformly at random among them. Since each of these events are independent we have $|\cS_v|\stackrel{\mathrm{d}}{=}\sum_{i=2}^{n} B_i$. Moreover, writing 
$W_{n,\eps}=\sum_{i=n-\lceil n^\eps \rceil}^n B_i$, we also have 
\[W_{n,\eps}\stackrel{\mathrm{d}}{=} |\cS_v\cap [n- \lceil n^\eps \rceil]|.\]
 
We now apply Bernstein's inequality (see, e.g., \cite{JLR}, Theorem 2.8) to obtain that for any $t>0$, 
\[\p{W_{n,\eps}\le \E{W_{n,\eps}}-t}\le \exp \left\{-\frac{t^2}{2\E{W_{n,\eps}}}\right\}.\]
We take $t=\E{W_{n,\eps}}-c\ln n$. 
Since \[\E{W_{n,\eps}}=\sum_{i=n-\lceil n^\eps \rceil}^n \frac{2}{i}=2(1-\eps)\ln n +O(1),\]
setting $\delta=2(1-\eps)-c>0$ we have $t=
\delta \ln n+O(1)$, so 
\[\p{|\cS_v \cap[n-\lceil n^\eps\rceil]|<c\ln n}= \p{ W_{n,\eps}\le \E{W_{n,\eps}}-t} = O(1)\cdot n^{-\delta^2/(4(1-\eps))}.
\]
Choosing $0<\beta< \delta^2/4(1-\eps)$, the result follows. 
\end{proof}

The following lemma is the last ingredient for \refP{core prop}. 

\begin{lem}\label{T}
Fix an integer $k\ge 2$ and let $\eps\in (0,1)$. Then, for $n$ large enough,  
\[\p{\tau_k\le n-\lceil n^\eps \rceil}\le \frac{2k^2}{\lceil n^\eps\rceil-1}.\] 
\end{lem}
\begin{proof}
By the definition of $\tau_k$, if $\tau_k>n-\lceil n^\eps \rceil$ then $\{a_i,b_i\} \not\subset [k]$ for all $1\le i\le  n-\lceil n^\eps \rceil$. The events that $\{a_i,b_i\}\not\subset [k]$ are independent for distinct $i$ and $\p{\{a_i,b_i\}\subset [k]}=\frac{k(k-1)}{(n+1-i)(n-i)}$, so we have that 
\begin{align*}
\p{\tau_k>n-\lceil n^\eps \rceil} 
= \prod_{i=1}^{n-\lceil n^\eps \rceil} \left(1-\frac{k(k-1)}{(n+1-i)(n-i)}\right) 
\ge 1-\sum_{i=1}^{n-\lceil n^\eps \rceil} \frac{2k^2}{(n-i)^2}
\end{align*}
The last inequality holds for $n$ large enough. Since $\sum_{j=m}^\infty j^{-2}\le \int_{m-1}^\infty x^{-2}dx=(m-1)^{-1}$, we get 
\[\p{\tau_k\le n-\lceil n^\eps \rceil}
\le \sum_{i=1}^{n-\lceil n^\eps \rceil} \frac{2k^2}{(n-i)^2}
\le \sum_{j=\lceil n^{\eps}\rceil}^\infty \frac{2k^2}{j^2} =\frac{2k^2}{\lceil n^\eps\rceil-1}.
\qedhere \]
\end{proof}

We finish this section with the proof of \refP{core prop}. 
 
\begin{proof}[Proof of \refP{core prop}]
Fix $c\in (0,2)$, $k\in \N$ and let $m_1,\ldots, m_k<c\ln n$ be positive integers. Let $\eps=(2-c)/4$ so that \refP{bound Ik} holds for some $\beta(c)=\beta(c,\eps)>0$. 
For $k=1$, the result follows from the equality in \refL{Upper bound} and \refP{bound Ik} since 
\[\p{|\cS_1|<m_1}\le \p{|\cS_1\cap [n-\lceil n^\eps\rceil]|<c\ln n}=o(n^{-\beta}).\]

For $k\ge 2$, the upper bound is likewise established immediately by \refL{Upper bound}. For the lower bound, letting $I=n-\lceil n^\eps\rceil$, by \refL{T} and \refP{bound Ik} we have 
\[\p{I<\tau_k, |\cS_v\cap [I]|\ge m_v, v\in [k]}
\ge 1-\p{I\ge \tau_k} -\sum_{v\in [k]} \p{|\cS_v\cap [I]|< m_v}
\ge 1-o(n^{-\alpha}), \]
where $\alpha<\min\{\beta, \eps\}$. By \refL{Lower bound}, it follows that 
\[\p{\deg(v)\ge m_v,\, v\in [k]}= 2^{-\sum_v m_v}(1+o(n^{-\alpha})),\]
as required.
\end{proof}

\section{Proof of \refP{moments}}\label{proofmoments}

By \refC{relabel cor} we can study vertex degrees in $T^{(n)}$ and derive conclusions about the variables $X_i^{(n)},X_{\ge i}^{(n)}$, $i\in \Z$. Recall that we write $\deg(v)=\deg_{T^{(n)}}(v)$, for $v\in [n]$.
\begin{lem}\label{inclusion-ex}
For any $k\in \N$ and integers $m_1,\ldots, m_k$, 
\[
 \p{\deg(u)=m_u,\, u\in [k]}
= \sum_{j=0}^k \sum_{\stackrel{S\subset [k]}{|S|=j}} (-1)^j 
\p{\deg(u)\ge m_u+\I{u\in S},\, u\in [k]}.
\]
Furthermore, for $k'\in \N$ and integers $m_{k+1},\ldots m_{k+k'}$,
\begin{align*}
&\p{\deg(u)=m_u,\, \deg(v)\ge m_v,\, 1\le u\le k<v\le k+k'}\\
=& \sum_{j=0}^k \sum_{\stackrel{S\subset [k]}{|S|=j}} (-1)^j 
\p{\deg(v)\ge m_v+\I{v\in S},\, v\in [k+k']}.
\end{align*}
\end{lem}
\begin{proof}
The second equation follows by intersecting the event $\{\deg(v)\ge m_v,\, k<v\le k+k'\}$ along all probabilities in the first equation. The first is straightforwardly proved using the inclusion-exclusion principle.  
\end{proof}

We are now ready to prove \refP{moments}. 

\begin{proof}[Proof of \refP{moments}]
Let $c\in (0,2)$ and $K\in \N$. Let $i<i'$ be integers such that $0<i+\log_n<i'+\log_n<c\ln n$ and let $a_j$, $i\le j\le i'$ be non-negative integers with $a_i+\cdots a_{i'}=K$. 
We are interested in the factorial moments $\E{(X_{\ge i'}^{(n)})_{a_{i'}} \prod_{i\le k<i'} (X_{k}^{(n)})_{a_k}}$. 

For $i\le k\le i'$, for each $v$ with $\sum_{l=i}^{k-1} a_l<v\le \sum_{l=i}^k a_l$ let $m_v=\lfloor \log n\rfloor +k$. Let  $K'=K-a_{i'}$, by \refC{relabel cor} and the exchangeability of the vertex degrees of $T^{(n)}$,
\begin{align*}
\E{(X_{\ge i'}^{(n)})_{a_{i'}} \prod_{i\le k<i'} (X_{k}^{(n)})_{a_k}}
&=(n)_K \p{\deg(u)=m_u,\, \deg(v)\ge m_v,\, 1\le u\le K'<v\le K}\\
&=(n)_K \sum_{l=0}^{K'} \sum_{\stackrel{S\subset [K']}{|S|=l}} (-1)^l 
\p{\deg(v)\ge m_v+\I{v\in S},\, v\in [K]},
\end{align*}
the last equality by \refL{inclusion-ex}. At this point we can apply \refP{core prop} to each of the terms. 
Since $m_v\le c\ln n$ for $v\in [K]$, there is $\alpha'=\alpha'(c,K)>0$ such that 
\begin{align*}
&\sum_{l=0}^{K'} \sum_{\stackrel{S\subset [K']}{|S|=l}} (-1)^l 
\p{\deg(v)\ge m_v+\I{v\in S},\, v\in [K]}\\
=& \sum_{l=0}^{K'} \sum_{\stackrel{S\subset [K']}{|S|=l}} (-1)^l 2^{-l-\sum_v m_v}(1+o(n^{-\alpha'}))\\
=& 2^{-\sum_v m_v} (1+o(n^{-\alpha'}))\sum_{l=0}^{K'} \sum_{\stackrel{S\subset [K']}{|S|=l}} (-1)^l 2^{-l}\\
=& 2^{-K'-\sum_v m_v}(1+o(n^{-\alpha'})).
\end{align*}
Using that $(n)_K=n^K(1+o(n^{-1}))$, we get
\[\E{(X_{\ge i'}^{(n)})_{a_{i'}} \prod_{i\le k<i'} (X_{k}^{(n)})_{a_k}}=2^{K\log n-K'-\sum_{v=1}^K m_v}(1+o(n^{-\alpha}));\]
where $\alpha=\min\{\alpha',1\}$. Finally, to complete the proof, note that 
\begin{align*}
K\log n-K'-\sum_{v=1}^K m_v
&=\sum_{v=K'+1}^K (\log n-m_v)+\sum_{v=1}^{K'} (\log n -1-m_v)\\
&=(-i'+\eps_n)a_{i'}+\sum_{k=i}^{i'-1} (-k-1+\eps_n)a_k.
\end{align*}
\vspace{-1cm}

\end{proof}

\section{Proofs of the main theorems}\label{proofs}

\begin{proof}[Proof of \refT{main}]
By Theorem 11.1.VII of \cite{DaleyVere-JonesII}, weak convergence in $\mathcal{M}_{\Z^*}^{\#}$ is equivalent to convergence of FDD's, that is, convergence of every finite family of bounded continuity sets; see Definition 11.1.IV of \cite{DaleyVere-JonesII}.
For any point process $\xi$ on $\Z$ and any $i\in \Z$, we have that $\Z\cap [i,\infty)$ is a bounded stochastic continuity set for the underlying measure of $\xi$ in $\mathcal{M}_{\Z^*}^{\#}$. Thus, any FDD of $\xi$ can be recovered from suitable marginals of the joint distribution of $(\xi(i),\ldots,\xi(i-1'), \xi[i,\infty))$ for some $i<i'\in \Z$.

Let $\eps\in [0,1]$ and $(n_l)_{l\ge 1}$ be an increasing sequence with $\eps_{n_l}\to \eps$. 
The goal then is to prove that, for any integers $i<i'$, the joint distribution of 
\[X_{i}^{(n_l)},\ldots, X_{i'-1}^{(n_l)}, X_{\ge i'}^{(n_l)}\]
 converges to the joint distribution of
\[\cP^{\eps}(i), \ldots, \cP^{\eps}(i'-1), \cP^{\eps}[i',\infty),\]
 that is, to the law of independent Poisson r.v.'s with parameters $2^{-i-1+\eps}, \ldots, 2^{-i'-2+\eps}, 2^{-i'+\eps}$. 

We compute the limit of the factorial moments of $X_{i}^{(n_l)},\ldots, X_{i'-1}^{(n_l)}, X_{\ge i'}^{(n_l)}$. For any non-negative integers $a_i,\ldots,a_{i'}$, by \refP{moments},
\begin{align*}
\E{(X_{\ge i'}^{(n)})_{a_{i'}} \prod_{i\le k<i'} (X_{k}^{(n)})_{a_k}}
&=\left(2^{-i'+\eps_n} \right)^{a_{i'}} \prod_{i\le k<i'} \left( 2^{-(k+1)+\eps_n} \right)^{a_k}(1+o(n^{-\alpha}))\\
&\to \left( 2^{-i'+\eps} \right)^{a_{i'}} \prod_{i\le k<i'} \left( 2^{-(k+1)+\eps} \right)^{a_k},
\end{align*}
as $n_l\to \infty$. The limit correspond to the factorial moment
\[\E{(\cP^{\eps}[i',\infty))_{a_i'} \prod_{i\le k<i'}(\cP^{\eps}(k))_{a_k} }.\]
The result follows (by, e.g. Theorem 6.10 of \cite{JLR}).
\end{proof}

\begin{proof}[Proof of \refT{max}]
Since $\{\Delta_n\ge \lfloor \log n\rfloor +i\}=\{X_{\ge i}^{(n)}>0\}$, we need only to estimate $\p{X_{\ge i}^{(n)}>0}$. 
If $i=O(1)$, then $\exp\{-2^{-i+\eps_n}\}=O(1)$ and so it suffices to prove that
\[\p{X_{\ge i}^{(n)}=0}-\exp\{-2^{-i+\eps_n}\}) \to 0,\]
as $n\to \infty$. This follows from \refT{main} and the subsubsequence principle. Suppose that there exists $\delta>0$ and a subsequence $n_k$ for which $|\p{X_{\ge i}^{(n_k)}=0}-\exp\{-2^{-i+\eps_{n_k}}\}|> \delta$. Since $\{\eps_{n_k}\}_{k\ge 1}$ is a bounded set there is a subsubsequence $n_{k_l}$ such that $\eps_{n_{k_l}}\to \eps$ for some $\eps\in [0,1]$. By \refT{main}, 
$\p{X_{\ge i}^{(n_{k_l})}=0}\to \exp\{-2^{-i+\eps}\}$;
this contradicts our assumption on the subsequence $n_k$. 

Now consider the case $i\to \infty$ with $i+\log_n <2\ln n$. By a standard inclusion-exclusion argument (see, e.g., \cite{Bollobas} Corollary 1.11), 
\begin{align}\label{Bollobas}
\p{X_{\ge i}^{(n)}=0}=\sum_{r=0}^n (-1)^r \frac{\E{(X_{\ge i}^{(n)})_r}}{r!},
\end{align}
and this sum has the so called \emph{alternating inequalities} property; this means that partial sums alternatively serve as upper and lower bounds for $\p{X_{\ge i}^{(n)}=0}$. Consequently
\footnote{A similar lower bound for $\p{X_{\ge i}^{(n)}>0}$ could be obtained from Paley-Zigmund's inequality.},
\begin{equation}\label{PZ}
\E{X_{\ge i}^{(n)}}- \frac{1}{2}\E{(X_{\ge i}^{(n)})_2}\le \p{X_{\ge i}^{(n)}>0}\le \E{X_{\ge i}^{(n)}}.
\end{equation}
Using \refP{moments} and the fact that $i\to \infty$, we have that $\E{X_{\ge i}^{(n)}}=2^{-i+\eps_n}(1+o(1))$ and 
\begin{align*}
\E{X_{\ge i}^{(n)}}-\frac{1}{2}\E{(X_{\ge i}^{(n)})_2}=2^{-i+\eps_n}(1+o(1))=(1-\exp\{-2^{-i+\eps_n}\})(1+o(1)). 
\end{align*}
The result follows. 
\end{proof}

\begin{proof}[Proof of \refT{normal}]
We again use the method of moments. By Theorem 1.24 of \cite{Bollobas}, it suffices to prove that, as $n\to \infty$
\begin{equation}\label{normal cond}
\E{(X_i^{(n)})_a}-(2^{-i-1+\eps_n})^a=o(2^ {-(i+1-\eps_n)b}),
\end{equation}
for all fixed $1\le a\le b$.
Since $i=o(\ln n)$, we have that $2^{-i-1+\eps_n}=n^{o(1)}$. On the other hand, by \refP{moments} there is $\alpha>0$ such that 
\begin{align*}
\E{(X_i^{(n)})_a}-(2^{-i-1+\eps_n})^a=o(n^{-\alpha}2^{-(i+\eps_n)a})=n^{-\alpha+o(1)}=o(n^{o(1)}).
\end{align*} 
Therefore, condition \eqref{normal cond} is satisfied and the proof is complete.
\end{proof}

{\bf Acknowledgements.} Laura Eslava would like to thank Henning Sulzbach for some very helpful discussions.


\begin{thebibliography}{14}
\providecommand{\natexlab}[1]{#1}
\providecommand{\url}[1]{\texttt{#1}}
\expandafter\ifx\csname urlstyle\endcsname\relax
  \providecommand{\doi}[1]{doi: #1}\else
  \providecommand{\doi}{doi: \begingroup \urlstyle{rm}\Url}\fi

\bibitem[Addario-Berry(2015)]{ab14}
Louigi Addario-Berry.
\newblock Partition functions of discrete coalescents: from {C}ayley's formula
  to {F}rieze's $\zeta(3)$ limit theorem.
\newblock In \emph{XI Symposium on Probability and Stochastic Processes},
  volume~68 of \emph{Progress in Probability}, Basel, 2015. Birkhauser.

\bibitem[Berestycki(2009)]{Berestycki09}
Nathana{{\"e}}l Berestycki.
\newblock \emph{Recent progress in coalescent theory}, volume~16 of
  \emph{Ensaios Matem{\'a}ticos [Mathematical Surveys]}.
\newblock Sociedade Brasileira de Matem{\'a}tica, Rio de Janeiro, 2009.
\newblock ISBN 978-85-85818-40-1.

\bibitem[Bollob{{\'a}}s(2001)]{Bollobas}
B{{\'e}}la Bollob{{\'a}}s.
\newblock \emph{Random graphs}, volume~73 of \emph{Cambridge Studies in
  Advanced Mathematics}.
\newblock Cambridge University Press, Cambridge, second edition, 2001.
\newblock ISBN 0-521-80920-7; 0-521-79722-5.
\newblock \doi{10.1017/CBO9780511814068}.
\newblock URL \url{http://dx.doi.org/10.1017/CBO9780511814068}.

\bibitem[Daley and Vere-Jones(2008)]{DaleyVere-JonesII}
D.~J. Daley and D.~Vere-Jones.
\newblock \emph{An introduction to the theory of point processes. {V}ol. {II}}.
\newblock Probability and its Applications (New York). Springer, New York,
  second edition, 2008.
\newblock ISBN 978-0-387-21337-8.
\newblock \doi{10.1007/978-0-387-49835-5}.
\newblock URL \url{http://dx.doi.org/10.1007/978-0-387-49835-5}.
\newblock General theory and structure.

\bibitem[Devroye(1987)]{Devroye87}
L.~Devroye.
\newblock Branching processes in the analysis of the heights of trees.
\newblock \emph{Acta Inform.}, 24\penalty0 (3):\penalty0 277--298, 1987.
\newblock ISSN 0001-5903.
\newblock \doi{10.1007/BF00265991}.
\newblock URL \url{http://dx.doi.org/10.1007/BF00265991}.

\bibitem[Devroye and Lu(1995)]{DevroyeLu95}
Luc Devroye and Jiang Lu.
\newblock The strong convergence of maximal degrees in uniform random recursive
  trees and dags.
\newblock \emph{Random Structures Algorithms}, 7\penalty0 (1):\penalty0 1--14,
  1995.
\newblock ISSN 1042-9832.
\newblock \doi{10.1002/rsa.3240070102}.
\newblock URL \url{http://dx.doi.org/10.1002/rsa.3240070102}.

\bibitem[Goh and Schmutz(2002)]{GohSchmutz02}
William Goh and Eric Schmutz.
\newblock Limit distribution for the maximum degree of a random recursive tree.
\newblock \emph{J. Comput. Appl. Math.}, 142\penalty0 (1):\penalty0 61--82,
  2002.
\newblock ISSN 0377-0427.
\newblock \doi{10.1016/S0377-0427(01)00460-5}.
\newblock URL \url{http://dx.doi.org/10.1016/S0377-0427(01)00460-5}.
\newblock Probabilistic methods in combinatorics and combinatorial
  optimization.

\bibitem[Janson(2005)]{Janson05}
Svante Janson.
\newblock Asymptotic degree distribution in random recursive trees.
\newblock \emph{Random Structures Algorithms}, 26\penalty0 (1-2):\penalty0
  69--83, 2005.
\newblock ISSN 1042-9832.
\newblock \doi{10.1002/rsa.20046}.
\newblock URL \url{http://dx.doi.org/10.1002/rsa.20046}.

\bibitem[Janson et~al.(2000)Janson, {\L}uczak, and Rucinski]{JLR}
Svante Janson, Tomasz {\L}uczak, and Andrzej Rucinski.
\newblock \emph{Random graphs}.
\newblock Wiley-Interscience Series in Discrete Mathematics and Optimization.
  Wiley-Interscience, New York, 2000.
\newblock ISBN 0-471-17541-2.
\newblock \doi{10.1002/9781118032718}.
\newblock URL \url{http://dx.doi.org/10.1002/9781118032718}.

\bibitem[Joag-Dev and Proschan(1983)]{Joag-DevProschan83}
Kumar Joag-Dev and Frank Proschan.
\newblock Negative association of random variables, with applications.
\newblock \emph{Ann. Statist.}, 11\penalty0 (1):\penalty0 286--295, 1983.
\newblock ISSN 0090-5364.
\newblock \doi{10.1214/aos/1176346079}.
\newblock URL \url{http://dx.doi.org/10.1214/aos/1176346079}.

\bibitem[Mahmoud and Smythe(1992)]{MahmoudSmythe92}
Hosam~M. Mahmoud and R.~T. Smythe.
\newblock Asymptotic joint normality of outdegrees of nodes in random recursive
  trees.
\newblock \emph{Random Structures Algorithms}, 3\penalty0 (3):\penalty0
  255--266, 1992.
\newblock ISSN 1042-9832.
\newblock \doi{10.1002/rsa.3240030305}.
\newblock URL \url{http://dx.doi.org/10.1002/rsa.3240030305}.

\bibitem[Meir and Moon (1988)]{MeirMoon88}
A. Meir and J.~W. Moon.
\newblock Recursive trees with no nodes of out-degree one. 
\newblock \emph{Congr. Numerantium} 66: 49--62,1988.

\bibitem[Na and Rapoport(1970)]{NaRapoport70}
Hwa~Sung Na and Anatol Rapoport.
\newblock Distribution of nodes of a tree by degree.
\newblock \emph{Math. Biosci.}, 6:\penalty0 313--329, 1970.
\newblock ISSN 0025-5564.

\bibitem[Pittel(1994)]{Pittel94}
Boris Pittel.
\newblock Note on the heights of random recursive trees and random {$m$}-ary
  search trees.
\newblock \emph{Random Structures Algorithms}, 5\penalty0 (2):\penalty0
  337--347, 1994.
\newblock ISSN 1042-9832.
\newblock \doi{10.1002/rsa.3240050207}.
\newblock URL \url{http://dx.doi.org/10.1002/rsa.3240050207}.

\bibitem[Szyma{\'n}ski(1990)]{Szymanski90}
Jerzy Szyma{\'n}ski.
\newblock On the maximum degree and the height of a random recursive tree.
\newblock In \emph{Random graphs '87 ({P}ozna\'n, 1987)}, pages 313--324.
  Wiley, Chichester, 1990.
\end{thebibliography}

\def\cprime{$'$}

\medskip

\end{document}